\documentclass[12pt]{article}
\usepackage[utf8]{inputenc}

\usepackage{amsmath,amsthm,amscd,amssymb,eucal,mathrsfs}
\usepackage{amsfonts}
\usepackage{latexsym}
\usepackage{mathabx}
\usepackage{graphicx} 
\usepackage{color}
\usepackage{multicol}
\usepackage{dsfont}
\usepackage[all]{xy}
\usepackage{tcolorbox}

\usepackage{scrextend}

\newcommand{\bohrc}[1]{\mathbb{B}}
\newcommand{\logz}[1]{\frac{\zeta'}{\zeta}}

\newtheorem{definition}{Definition}
\newtheorem{remark}{Remark}
\newtheorem{theorem}{Theorem}
\newtheorem{lemma}{Lemma}
\newtheorem{corollary}{Corollary}
\newtheorem{proposition}{Proposition}
\newtheorem{example}{Example}
\title{A probabilistic interpretation of Weil's explicit sums and arithmetic spectral measures. }
\author{\'Angel Alfredo Mor\'an Ledezma}
\date{\today}

\begin{document}

\maketitle

\begin{abstract}
    In this paper we study the connections of three paradigms in number theory: the adelic formulation of the Riemann zeta function, the Weil explicit formula and the concepts of the so called probabilistic number theory initiated by Harald Bohr. We give a different reformulation, rooted in the adelic framework, of the theory of distribution values of the Riemann zeta function. By introducing the Bohr compactification of the real numbers as a natural probability space for this theory, we show that the Weil explicit sum can be expressed in terms of covariances and expected values attached to random variables defined on this space. Moreover, we express the explicit formula as a limit of spectral integrals attached to operators defined on the Hilbert space of square-integrable functions on the Bohr compactification. This gives a probabilistic and a geometrical interpretation of the Weil explicit formula. 
\end{abstract}

%%%%%%%%%%%%%%%%%%%%%%%%%%%%%%%%%%%%%%
\section{Introduction.}
Many different realms of mathematics are involved in the description of the behavior of the Riemann zeta function. One of them is probability theory. This branch of mathematics, the so called probabilistic number theory has different definitions for different authors, here we will understand it as the study of the distribution values of arithmetic objects such as Dirichlet series. This theory has been employed successfully to study the analytic properties of a large class of $L$-functions, for example, the so called universality property of zeta functions. In particular the following problem is usually investigated: given some set, how often do values of an arithmetic object get into this set? This frequency follow strict mathematical laws, and as stated by L. Laurinčikas  in \cite{LimitLaurincikas}, this give an analogy with quantum mechanics where it is impossible to describe the chaotic behavior of one particle (in this case concrete values of the zeta function), but a large number of particles obey statistical laws. H. Bohr devoted his scientific carrier in the study of such distributions. Together with his student B. Jessen were able to describe in probabilistic terms the complicated behavior of the function $t\mapsto \log \zeta(\sigma+it)$ for $t\in \mathbb{R}$ and and $\sigma>\frac{1}{2}$, where $\sigma+it $ lies in the zero-free region of $\zeta$ in the strip $\frac{1}{2}<\sigma\leq 1$ (see Theorem 1 and Theorem 2), this result is usually called the Bohr-Jessen limit theorem and describes the limit behavior as $T\rightarrow \infty$, of the probability measures given by $\frac{1}{2T}\mu_0(\{t\in [-T,T]: \zeta(\sigma+it)\in A)$, where $A\subset{\mathbb{C}}$ is open and $\mu_0$ denotes the usual Lebesgue measure on the real numbers. In other words, one is interested in the asymptotic behavior of the complex random variable $t\mapsto \zeta(\sigma+it)$, defined on the probability space $([-T,T],\mu_0/2T)$ \cite{Bohrjessenlimit}.  In this work we investigate how this theory is connected with other two paradigms in number theory: the Iwasawa-Tate theory, which allow us to represent the zeta function as an adelic integral over certain idele group, and the well know Weil explicit sums, which encapsulate the relationship between the distribution of the prime numbers on the real line and the zeroes of the zeta function. This last two subjects are already connected. Works that shows this connection include the results of S. Haran, for example in \cite{Haran1990} and by the well known results of A. Connes \cite{Connes1997}. Both of them achieve to obtain different geometric expressions for the Weil explicit sums as an additive convolution with the so called Riesz potential and as Trace formula, respectively. This results can be considered as part of the framework of the Polya-Hilbert project on the study of the Riemann Hypothesis. This work aims to contribute in two different ways: one is exhibit explicitly the connections with the Bohr-Jessen probability theory and the adelic theory of the Riemann zeta function by presenting a new approach of the formulation of how we can understand the limit distribution generated by the random variable $t\mapsto \zeta(\sigma+it)$ in both context (adelic and probabilistic), this is explained in the first chapter and is motivated by an observation of S. Haran given in \cite{Mystery}. The second one is the interpretation of the Weil explicit sums in this probabilistic framework. By using the concept of "off-diagonal" zeta function, we present the Bohr compactification of the reals, denoted by $\bohrc{}$ (and isomorphic to the non-necesarilly continuous homomorphisms $Hom[\mathbb{R},S^1]$ from the additive group of the real numbers to the complex unit circle)  as the natural framework to develop the known theory of probabilistic number theory. Remember that for $Re(s)>1$ the complete Riemann zeta function can be understood as an adelic integral taking values on $\mathbb{C}=Hom_{cont}(\mathbb{A}^{*}/\mathbb{Z}^*_{\mathbb{A}}\mathbb{Q}^*,\mathbb{C}^*)$. This space is embedded diagonally on $Hom_{cont}(\mathbb{A}^*/\mathbb{Z}_{\mathbb{A}}^*,\mathbb{C}^*)$. The "off-diagonal" functions that appeared in this work can be understood as extensions of finite-adelic integrals taking values on a compact subdomain of $Hom_{cont}(\mathbb{A}^*/\mathbb{Z}_{\mathbb{A}}^*,\mathbb{C}^*)$. It turns out that this functions can be considered as random variables on the Bohr compactification of the reals, and as a first result we show in Theorem 4 the connection of the "off-diagonal" Riemann zeta function and the Bohr-Jessen limit theorem. Therefore we expect that this observation gives a natural reformulation rooted on the Adelic framework of the works of Bohr, Jessen, Matsumoto, Laurincickas, Takanobu and others (see for example \cite{takanobu,LimitLaurincikas,Bohrjessenlimit,AntanasLaurinčikas2001}). The second part of this paper is devoted to the study of the Weil explicit sums. It was shown by A. Weil \cite{Weilexplicit} that the Riemann Hypothesis is equivalent to the negativity of the right-hand side of the Weil explicit formula:  
\begin{equation*}
    \begin{split}
        M(f*\bar{f}^*)(0)+&M(f*\bar{f}^*)(1)-\sum_{\zeta_{\mathbb{Q}}(s)=0} M(f*\bar{f}^*)(s)\\
        &=\sum_{p\leq \infty} \int_{1/2-i\infty}^{1/2+i\infty} M(f*\bar{f}^*)(s)d\log\frac{\zeta_p(1-s)}{\zeta_p(s)},
    \end{split}
\end{equation*}
where $\zeta_p(s)=(1-p^{-s})^{-1}$ for prime $p<\infty$, $\zeta_{\infty}(s)=\pi^{-s/2}\Gamma(\frac{s}{2})$ and $\zeta_{\mathbb{Q}}$ is the complete Riemann zeta function. Here, $f\in C_{c}^{\infty}(\mathbb{R}_{+}^*)$ and $M(f)(s)=\int_{0}^{\infty}f(x)x^{s} \frac{dx}{x}$ is the Mellin transform of $f$. The right hand side, denoted by $W(f*\bar{f}^*)$ shows the relationship between the zeroes of the zeta function and the prime numbers. In order to study this objetct in the probability framework first we show that the $\mathbb{Z}_p^*$-invariant functions of  $L^2(\mathbb{Q}_p^*)$  form a subspace of $L^2(\mathbb{B})$ where the $p$-adic Mellin transform gives the isometry. Moreover, using the structure of $C^*$-algebra of $C(\mathbb{B})$, the space of continuous functions on $\mathbb{B}$, we show that $L^1(\mathbb{Q}_p^*)^{\mathbb{Z}_p^*}$ can be embedded in $C(\mathbb{B})$ using the Mellin transform preserving the compatibility of  involutions of $*-$algebras. This allow us to construct many different random variables, in particular, for any function on $f\in C_{c}^{\infty}(\mathbb{R}_+^*)$ its restriction to the lattice generated by $\{p^k:p \ is \ prime, \ k\in \mathbb{Z}\}$ generate a well defined function on $L^1(\mathbb{Q}_p^*)^{\mathbb{Z}_p^*}\cap L^2(\mathbb{Q}_p^*)^{\mathbb{Z}_p^*}$, for very prime $p$, and therefore a random variable in $C(\mathbb{B})$. With this we are able to show that the local terms in the Weil explicit sum are in fact a covariance (also an expected value) of the random variable generated by $f$ and the random variable generated by $Re \ \frac{d}{ds}\log \zeta_p(s)$, the real part of the logarithmic derivative of the local factor of the Riemann zeta function (Theorem 5 and Theorem 6). Moreover, since $L^2(\mathbb{B})$ provides a faithful representation of the $C^*-$algebra $C(\bohrc{})$, this covariance (and therefore the finite part of the Weil explicit sum, $p<\infty$) can be interpreted as an spectral integral generated by the multiplication operator attached to the real part of the  "off-diagonal" logarithmic derivative of the Riemann zeta function $Re \frac{\zeta'}{\zeta}$ restricted on the critical line $Re(s)=\frac{1}{2}$. With this we get two types of descriptions, one of probabilistic nature and one of geometrical nature. Chapter 4 is devoted to the extension of this results now considering also the archimedean place. For this, inspired by the ideas of \cite{Mystery} and \cite{connespositive} we use the $q$-discretization of the gamma function to give an approximation of the complete Weil explicit sum. We show that the so called $q$-calculus appeared naturally in the spaces $L^2(\bohrc{B})$ and $C(\bohrc{})$. This allow us to achieve a full description of the Weil explicit sum in terms of the probabilistic framework of the probability space $(\mathbb{B},d\chi)$ where $d\chi$ denotes the Haar measure of $\mathbb{B}$, therefore the Riemann hypothesis becomes equivalent to an upper bound to certain covariance or the negativity of certain expected value (Corollary 1). Moreover, we show that $W(f*\bar{f}^*)$ can be understood as a limit of expected values and as a limit of spectral integrals of operators defined on $L^2(\bohrc{})$ (Theorem 8, Corollary 1, Theorem 9, Corollary 2). The analytical structure of $C(\bohrc{})$ as a $C^*$-algebra and its connection with almost periodic functions, is already very well known by the physics community since this space appears naturally in the description of the theory of Polymer Quantum Mechanics (PQM) \cite{bohrrelativ,Velhinho_2007,FLORESGONZALEZ2013394}. In chapter 5 we present some outlook regarding this connection with PQM, and proposing this theory as a possible alternative way to explore the ideas of the Polya-Hilbert project. 

\section{Preliminaries: The diagonal zeta function.}
In this section we present one important observation given by S. Haran in \cite{Mystery}, which allow us to introduce the main idea behind the results of this work. The goal is to understand the relationships between three paradigms in Number Theory, namely, the ideas of Tate-Iwasawa Theory, The Weil explicit sums and the Bohr-Jessen  probability theory of the zeta function. \newline

We are interested in the unramiefied part of Tate-Iwasawa theory, that is on the on the $\mathbb{Z}_p^*-$invariant theory (locally), where $\mathbb{Z}_p^*$ denotes the multiplicative group of the $p$-adic integers. Denote by $\mathbb{Q}_p^*$ the multiplicative group of $\mathbb{Q}_p$ with the attached  Haar measure $d^*\mu_p$ , also we have $\mathbb{Z}_p^{*}=\{x\in \mathbb{Z}_p: |x|_p=1\}$, where $|\cdot|_p$ is the canonical absolute value on $\mathbb{Q}_p$. We have that 
\[\mathbb{Q}_p^* / \mathbb{Z}_p^*=p^{{\mathbb{Z}}},\]
By taking the corresponding Pontryagin dual we obtain
\[\widehat{\mathbb{Q}_p^* / \mathbb{Z}_p^*}=i\mathbb{R}/\frac{2\pi i}{\log p}\mathbb{Z},\]
That is, any element of $Hom_{cont}(\mathbb{Q}_p^* / \mathbb{Z}_p^*,S^1)$ is of the form $|\cdot|_p^{it}$. These correspond to the unramified characters on $\mathbb{Q}_p^*$. More generally for unramified quasi-characters
\[Hom_{cont}(\mathbb{Q}_p^* / \mathbb{Z}_p^*,\mathbb{C}^*)=\mathbb{C}/\frac{2\pi i}{\log p} \mathbb{Z},\]
that is, an element $\chi$ in this space is of the form $\chi=|\cdot|_p^{\sigma+it}$.  

For $f(x)=f(|x|_p)\in C_c(\mathbb{Q}_p^*)^{\mathbb{Z}_p^*}$,  we define its \textit{local Mellin transform} at $\chi=|\cdot|^{\sigma+it}$ or equivalently at $s=\sigma+it$ as the multiplicative Fourier transform of $f|\cdot|_p^{\sigma}$, i.e.,
\[\mathcal{M}_p(f)(s):=\int_{\mathbb{Q}_p^{*}} f(|x|_p)|x|_p^{s}d^*\mu_p(x)=\sum_{k\in \mathbb{Z}}f(p^{-k})p^{-ks}.\]
    The local factor of the zeta function is then $\zeta_p(s)=\mathcal{M}_p(\phi_p)(s)=(1-p^{-s})^{-1}$, where $\phi_p(x)=1_{\mathbb{Z}_p}(x)$, the indicator function of $\mathbb{Z}_p$ . In his Thesis \cite{Tate1950FourierAI}, John Tate makes the following observation, \newline
\begin{addmargin}[1em]{1em}
\begin{center}
\emph{"We can do nothing really significant with the idele group until we imbed the multiplicative
group $K^*$ of K in it ”.\\
\hspace{\fill}— John Tate}
\end{center}
\end{addmargin}
As stated by J. Tate, in the Global Theory, we are not interested in all the characters of the Idèle group $\mathbb{A}^*$. This is resumed in the following description of the quasi-characters of the Idèle class group $\mathbb{A}^*/\mathbb{Q}^*$:

\[Hom_{cont}(\mathbb{A}^{*}/\mathbb{Q}^*,\mathbb{C}^*)\ni\chi=\eta(\cdot)|\cdot|_{\mathbb{A}}^{\sigma+it},\]
where $\eta$ is the ramified part of the character $\chi$, or in the unramified case:
\[\mathbb{C}=Hom_{cont}(\mathbb{A}^{*}/\mathbb{Z}^*_{\mathbb{A}}\mathbb{Q}^*,\mathbb{C}^*)\ni\chi=|\cdot|_{\mathbb{A}}^{\sigma+it},\]
where $\mathbb{Z}_{\mathbb{A}}^*$denotes the group of units in the ring of Adele integers. Therefore, for a sufficiently nice function $f$, its \textit{Global Mellin transform} is defined at the quasi-character $\chi=|\cdot|_{\mathbb{A}}^{s}$, or equivalently at $s=\sigma+it$ by the integral 
\[\mathcal{M}(f)(\chi):= \int_{\mathbb{A}^*} f(x) \chi(x) d^*\mu_{\mathbb{A}}(x), \]
In particular for the function $\phi=\otimes_{p\leq \infty}\phi_p$, where $\phi_{\infty}=e^{-\pi x^2}$, and for fixed $Re(s)=\sigma>1$, its global Mellin transform gives the complete Riemann zeta function
\[\zeta_{\mathbb{Q}}(s)=\mathcal{M}(\phi)(s)=\int_{\mathbb{A}^*}\phi(x)|x|_{\mathbb{A}}^s d^*\mu_{\mathbb{A}^*}(x),\]
With the convergent Euler product
\[\zeta_{\mathbb{Q}}(s)=\prod_{p\leq \infty} \zeta_p(s).\]
This leads us to understand the function  $\chi \mapsto \zeta_{\mathbb{Q}}(\chi)$ as a restriction on the diagonal generated by the following embedding: 
\[\mathbb{C}=Hom_{cont}(\mathbb{A}^{*}/\mathbb{Z}^*_{\mathbb{A}}\mathbb{Q}^*,\mathbb{C}^*)\hookrightarrow Hom_{cont}(\mathbb{A}^*/\mathbb{Z}_{\mathbb{A}}^*,\mathbb{C}^*)\]
where
\[Hom_{cont}(\mathbb{A}^*/\mathbb{Z}_{\mathbb{A}}^*,\mathbb{C}^*)=\mathbb{C}\times\prod_{p<\infty}\mathbb{C}/\frac{2\pi i}{\log p}\mathbb{Z},
\]
Therefore we can rewrite the Euler product as: 
\[\zeta_{\mathbb{Q}}(s)=\prod_{p\leq \infty} \zeta_p(s_p)|_{s_p=s}.\]
\subsection{The off-diagonal zeta function as a random variable.}
We now extend the above observation given by S. Haran, we present a way to understand the off diagonal zeta function. The Riemann zeta function is defined by the contribution of the finite primes:
\[s\mapsto\zeta(s):=\prod_{p<\infty}\zeta_p(s),\]
for $Re(s)>1$. As the case above, this can be seen as the restriction (of a more general function) to the diagonal $\mathbb{C}$. In order to give a formal description of this, consider the space 
\[Hom_{cont}(\mathbb{A}_f^*/\hat{\mathbb{Z}}^*,\mathbb{C}^*)=\prod_{p<\infty} \mathbb{C}/\frac{2\pi i}{\log p}\mathbb{Z}\]
where $\mathbb{A}_f^*$ denotes the multiplicative ring of finite Adèles and $\hat{\mathbb{Z}}$ the ring of profinite integers. For our analysis, we are interested in the behavior of the Riemann zeta function in the region $Re(s)=\sigma$ , hence we fix a real part $\sigma>0$ and consider the Pontryiagin dual instead: 
\[\Omega:=Hom_{cont}(\mathbb{A}_f^*/\hat{\mathbb{Z}}^*,S^1)=\prod_{p<\infty} i\mathbb{R}/\frac{2\pi i}{\log p}\mathbb{Z}.\]
By Tychonoff's Theorem, $\Omega$ is a compact abelian group equipped with coordinatewise multiplication. Then we can take the \textit{off-diagonal zeta function} defined by (formally)
\begin{equation}
    \Omega\ni\chi\mapsto \prod_{p<\infty} \zeta_p(\sigma,\chi)=:\zeta(\sigma,\chi),
\end{equation}
where the expression, $\zeta_p(\sigma,\chi)=\int_{\mathbb{Z}_p}|x|_p^{\sigma}\chi(x)dx$ denotes the local zeta function taking values on $\widehat{\mathbb{Q}_p^* / \mathbb{Z}_p^*}$, that is, quasi characters with fixed real part $Re(\chi)=\sigma$. Since there exists a probability measure on the space $\Omega$,  we can ask about the behavior of $(1)$ in probabilistic terms. We have the following Lemma. 
\begin{lemma}
    The product (1) converges almost surely. The assignation $\chi\mapsto \zeta(\sigma,\chi)$ is a well defined complex-valued random variable on the space $\Omega$. 
\end{lemma}
\begin{proof}
    The proof follows from the proof of Lemma 5.1.7. in \cite{LimitLaurincikas}. The main observation is to notice that the space $\Omega$ is isomorphic to the space $\Omega'$ defined below.  
\end{proof}
In fact, product $(1)$ was already considered by Bagchi and later by Laurincickas and other authors in order to describe the property of universality of the Riemann zeta function in probabilistic terms. For more details, the reader may consult for example \cite{Bagchi,LimitLaurincikas}.  Nevertheless they presented this function in an equivalent way: let $S^1$ be the complex unit circle, and consider the copies indexed by the prime numbers $\{S^1_p\}$. Let 
\begin{equation} \Omega^{'}:=\prod_{p}S^1_p,
\end{equation}
which, by Tychonoff's Theorem, is a compact Abelian group equipped with coordinate-wise multiplication. Hence, there exists a unique probability measure on the $\sigma$-algebra of Borel sets. Define the assignment
\begin{equation}
    \omega \mapsto \zeta(\sigma,\omega)=\prod_{p<\infty } \left(1+\sum_{k=1}^{\infty}\frac{\omega(p)^{k}}{p^{k\sigma}}\right),
\end{equation}
where $\omega\mapsto \omega(p)$ is the canonical projection in the coordinate $p$. This is an almost surely convergent product and a well-defined $\mathbb{C}$-valued random element defined on $\Omega'$ (this is a slight variation of Lemma 5.1.7. in \cite{LimitLaurincikas}). Moreover, the following holds: 
\begin{equation}
    \zeta(\sigma,\omega)=\sum_{k=1}^{\infty} \omega(k)k^{-\sigma},
\end{equation}
where $\omega(k)$ its defined by $\omega(k):=\prod_{p}\omega(p)^{\alpha(p)}$, and $k=\prod_{p}p^{\alpha(p)}$. 
As we will see, the "off-diagonal zeta function" $(1)$ will allow us to give a probability interpretation of the Weil explicit formula which will introduced in the next chapter.

\subsection{Weil Explicit Formula}
In \cite{Riemannorig}, Riemann shows a fundamental result on how the prime numbers and the zeros of the Riemann zeta function are related. An "Explicit formula" was introduced to study the prime numbers less than a given magnitude. This formula was made rigorous and generalized by many authors, a particularly formidable formula was rigorously proved by von Mangoldt: 

\begin{equation}
    \sum_{1<n<x} \Lambda(n) + \frac{1}{2} \Lambda(x)=x-\sum_{\rho}\frac{x^{\rho}}{\rho} - \log(2\pi)-\frac{1}{2}\log(1-x^{2}),
\end{equation}
where $x>1$ (not necesarily an integer), and $\Lambda(y)=log(p)$ if $y>1$ is a positive power of the prime number $p$, and $0$ otherwise. The number $\rho$ runs through the zeros of Riemann zeta function. Although Riemann already knew the exact relation between the zeros of the zeta function and the distribution of the primes, it was Weil who "crystallized" this in his explicit sums (formula) \cite{Weilexplicit}. Given $f:\mathbb{R}_{+}\rightarrow \mathbb{R}$ smooth and compactly supported, we associate with $f$ its Mellin transform $M(f)(s)=\int_{0}^{\infty}f(x)x^{s} \frac{dx}{x}$. We have by the functional equation $\zeta_{\mathbb{Q}}(s)=\zeta_{\mathbb{Q}}(1-s)$:
\begin{equation}
    \begin{split}
        M(f)(0)+&M(f)(1)-\sum_{\zeta_{\mathbb{Q}}(s)=0} M(f)(s)=-\frac{1}{2\pi i} \oint M(f)(s)d\log\zeta_{\mathbb{Q}}(s)\\
        &=\sum_{p\leq \infty} \int_{1/2-i\infty}^{1/2+i\infty} M(f)(s)d\log\frac{\zeta_p(1-s)}{\zeta_p(s)},
    \end{split}
\end{equation}
where $\zeta_p(s)=(1-p^{-s})^{-1}$ for prime $p<\infty$, and $\zeta_{\infty}(s)=\pi^{-s/2}\Gamma(\frac{s}{2})$. Denote by 
\[W_p(f):= \int_{1/2-i\infty}^{1/2+i\infty} M(f)(s)d\log\frac{\zeta_p(1-s)}{\zeta_p(s)}=\log p \sum_{k\neq 0} f(p^k) \min(p^k,1),\]$$W_{\infty}(f):=\int_{1/2-i\infty}^{1/2+i\infty} M(f)(s)d\log\frac{\zeta_{\infty}(1-s)}{\zeta_{\infty}(s)},$$
and the \textit{finite Weil sum},
$$W_{fin}(f):=\sum_{p<\infty}W_p(f).$$
The Riemann Hypothesis is equivalent to the Weil Negativity Criterion of the \textit{Weil sum}:
\[W(f):=\sum_{p\leq \infty}W_p(f)\leq 0, \]
for complex-valued test functions of the positive half-line of the form 
\[f(x)=\int_{0}^{\infty} g(xy)\overline{g}(y)dy, \hspace{0.3cm} M(g)(1)=0, \ M(g)(0)=0. \]
    In fact, it is enough to prove the negativity of $W$ for functions $g$ with compact support in the (locally compact) multiplicative group $\mathbb{R}^{*}_{+}=(0,\infty)$. Furthermore, given any set of complex numbers $F\supset \{0,1\}$ with $ F\cap \{s:\xi(s)=0\}=\emptyset,$ we have that the Riemann hypothesis is equivalent to 
\[W(g*\overline{g}^{*})\leq 0, \]
for all $g\in C_{c}^{\infty}(\mathbb{R}_+^{*})$, such that $M(g)(s)=0$  for all $s\in F$ and $g^{*}(x)=\frac{1}{x} g(\frac{1}{x})$, see \cite{Yoshida,connespositive}. Following \cite{connespositive}, it is worth to notice that $W(f)$ only involves a finite set of primes when evaluated on a test function with compact support. That is, even though the Riemann Hypothesis is about the asymptotic distribution of the primes, this equivalent formulation involves finitely many primes at a time. With this in mind, for functions with support contained in $(1/2,2)$, we have that $W(f)=W_{\infty}(f)$. In this case, in \cite{Yoshida} and \cite{connespositive} showed that $W_{\infty}(f)\leq 0$ . In particular, in \cite{connespositive}, Connes and Consani give operator theoretic conceptual reasons for Weil's criterion. Our approach is different. One goal is to show the natural link between the $C^*$-algebra $C(\bohrc{})$ of continuous complex-valued functions over $\bohrc{}$, the Bohr compactification of $\mathbb{R}$, the probabilistic interpretation of almost periodic functions, and the Weil sum $W(f)$. 

\subsection{The Bohr Compactification}
Consider the real line $\mathbb{R}$ equipped with the additive commutative group structure. The Bohr compactification $\bohrc{}$ can be described as the set $Hom[ \mathbb{R},S^1]$ of all, not necessarily continuous, group morphisms from the group $\mathbb{R}$ to the multiplicative group $S^1$. Equivalently we can describe $\bohrc{}$ by the following 
\begin{definition}
    The spectrum of a unital Banach algebra $\mathcal{A}$ is the set of all non-zero $*$-homomorphisims $\chi:\mathcal{A}\rightarrow \mathbb{C}; a\mapsto \chi(a)$, called the characters.
\end{definition}
\begin{definition}
\begin{enumerate}
    \item For any $k\in \mathbb{R}$ define the periodic functions of period $2\pi / k$ by 
    \[T_k:\mathbb{R}\rightarrow \mathbb{C}; x\mapsto e^{ikx},\]
The $*$-algebra $\mathcal{C}$ of almost periodic functions is the finite complex linear span of the functions $T_k$, that is, functions of the form 
\begin{center}
$f=\sum_{I=1}^{N}z_IT_{k_I} $ where $N<\infty$, $k_I\in \mathbb{R}$, $z_I\in \mathbb{C}$
\end{center}
    \item Let $\overline{\mathcal{C}}$ be the closure of $\mathcal{C}$ with respect the sup-norm on $\mathbb{R}$. This is an abelian $C^*$-algebra with respect to pointwise operations and complex conjugation as involution. The spectrum of this algebra denoted by $\bohrc{}$ is called the Bohr compactification of $\mathbb{R}$.
\end{enumerate}  
\end{definition}
For more details the reader may consult \cite{bohrabstract,bohrgelfad,bohrrelativ}. Let $\chi \in \bohrc{}$, that is, an arbitrary homomorphism $\chi:\overline{\mathcal{C}}\rightarrow \mathbb{C}$ without any continuity assumption. Any such character is determined once we know its values $X(k):=\chi(T_k)$. Since 
\begin{equation}
    X(k)X(l)=X(k+l), \ \overline{X(k)}=X(-k),
\end{equation}
we have that $X:\mathbb{R}\rightarrow S^1$ is a group homomorphism which does not need to be continuous, that is $\bohrc{}=Hom[ \mathbb{R},S^1]$ as stated. The above compactification can be extended to any topological group, moreover it is worth noticing the following. Let $G=\mathbb{R}$ be the topological group of the real numbers with addition and the euclidean topology. Consider the Pontryagin dual of $G$ equipped with the discrete topology denoted by $\hat{G}_d$. Then the Pontryagin dual of this group coincides with $\bohrc{}$.  Hence, it is clear that for any self dual group $G$, its Bohr compactification is given by $Hom[G,S^1]$. 
\begin{remark}
    The space $\overline{\mathcal{C}}$ is also known as the set of uniformly almost periodic functions over $\mathbb{R}$. In fact any such function can be seen as the restriction $f|_{\mathbb{R}}$ of $f\in C(\bohrc{})$, the space of continuous functions over $\bohrc{}$. Moreover, the operator $f\mapsto f|_{\mathbb{R}}$ is an isometric $*-$isomorphisim of $C(\bohrc{})$ onto $\overline{\mathcal{C}}$.
\end{remark}

Let $f\in \overline{\mathcal{C}}$,  then we will denote by $\hat{f}:\bohrc{}\rightarrow \mathbb{C}$ the associated lift function given by $\hat{f}(\chi)=\chi(f)$ for every $\chi \in \bohrc{}$.  In particular, the maps $\hat{T}_k$  will represent the map $\chi \mapsto \chi(T_k)$. For a set of functions $A\subset \overline{\mathcal{C}}$, we set $\hat{A}=\{\hat{f}:f\in A\}$. The topology of $\bohrc{}$ can be described in different ways. For our purposes we will use different equivalent ways to describe said topology. Since $Hom[ \mathbb{R},S^1]$ is a closed group of $\prod_{k\in\mathbb{R}} S^1_k$, where $S^1_k=S^1$, and since this product is compact by Tychonoff's Theorem, we have that the induced topology on $\bohrc{}$ makes this space a compact topological space. Moreover, $\bohrc{}$ is a compact Hausdorff group, which implies the existence of a unique normalized Haar measure denoted by $d\chi$.  The construction of this measure can be realized using the Riesz Representation Theorem in the following way.  For any $f\in \overline{\mathcal{C}}$  and $x\in \mathbb{R}$ we have the following map $x\mapsto \chi_{x}(\cdot)$ where $\chi_x(f)=f(x)$. This map embeds the group $\mathbb{R}$ densely into the group $\bohrc{}$. Moreover, since the attached homomorphism $X\in Hom[\mathbb{R},S^1]$ is given by $X(k)=\chi_{x}(T_k)=e^{ikx}$. Therefore, by using (1), $\mathbb{R}$ is embedded into $\bohrc{}$ as the set of smooth homomorphisms. \newline

Let $S=\{T_k: k\in \mathbb{R}\}$ be the total set generating $\overline{\mathcal{C}}$. Then, by the Stone-Weierstrass Theorem, the algebra $span \ \widehat{S}$ is dense in $C(\bohrc{})$, since $\widehat{S}$ separates points and contains the constant functions.  Let $\Lambda :span \ \widehat{S} \subset C(\bohrc{}) \mapsto \mathbb{R}$ , be the positive linear functional given by 
\[\Lambda(\hat{T_{k}})= \delta_{k,0},\]
where $\delta$ denotes the Kronecker delta. By the Riesz-Markov representation Theorem, there is a unique Radon measure (which can be easily proved to be the Haar measure with the standard normalization), such that 
\[\Lambda(F)=\int_{\bohrc{}} F(\chi) d\chi. \]
This implies that the set $\widehat{S}$  is an orthonormal basis of $L^2(\bohrc{},d\chi)$, and since the function $\textbf{1}: \bohrc{} \rightarrow \mathbb{R}$, $\textbf{1}(\chi)=1$ is given by $\textbf{1}=\hat{T}_{0}$, the measure of $\bohrc{}$ is $1$. Therefore, the space $(\bohrc{},\mathcal{B}, d\chi)$ is a probability space, where $\mathcal{B}$ is the $\sigma$- algebra of Borel sets of $\bohrc{}$. Henceforth we may refer to the probability space $(\bohrc{},\mathcal{B}, d\chi)$ as the Bohr space.  The following lemma relates the Haar measure of $\bohrc{}$ and the usual Lebesgue measure of $\mathbb{R}$. For the proof the reader may consult \cite[Equation  28.2.8]{bohrrelativ}.

\begin{lemma}
    Let $f\in C(\bohrc{}) $, then 
    $$\int_{\bohrc{}} f(\chi)d\chi=\lim_{T\rightarrow +\infty} \frac{1}{2T} \int_{-T}^{T} f|_{\mathbb{R}}(x) dx,$$
where $dx$ is the Lebesgue measure of $\mathbb{R}$. 
\end{lemma}
In particular for any $k$-periodic continuous function $f(t)$, we have 
\[\int_{\bohrc{}} \hat{f}(\chi)d\chi= \lim_{T\rightarrow +\infty} \frac{1}{2T} \int_{-T}^{T} f(x) dx= \frac{k}{2\pi} \int_{0}^{2\pi / k} f(t)dt. \]

\section{An arithmetic random process and Weil Explicit Sums.} 
\subsection{The $L^1(\mathbb{Q}_p^*)^{\mathbb{Z}_p^*}$ and $L^2(\mathbb{Q}_p^*)^{\mathbb{Z}_p^*}$ embedding.}
Let $L^2(\mathbb{Q}_p^*)^{\mathbb{Z}_p^*}=L^2(\mathbb{Q}_p^*/\mathbb{Z}_p^*)$, which by the Plancherel Theorem is isomorphic to $L^2(i\mathbb{R}/\frac{2\pi i}{\log p}\mathbb{Z})$, the later can be identified with the span of the functions $e^{ikt\log p},$ for $k\in \mathbb{Z}$. These functions can be identified with the functions $\hat{T}_{\log p^k}\in L^2(\bohrc{})$. Therefore the space $L^2(\mathbb{Q}_p^*)^{\mathbb{Z}_p^*}$ is isomorphic with the span of the functions of the form $\hat{T}_{\log p^k}$ for $k\in \mathbb{Z}$, a subspace of $L^2(\bohrc{})$. Moreover, the Mellin transform for real part $\sigma=0$ gives a linear isometry, that is, for $f\in L^2(\mathbb{Q}_p^*)^{\mathbb{Z}_p^*}$  
\begin{equation*}
    \begin{split}
        \mathcal{M}_p(f)(it)&=\int_{\mathbb{Q}_p^*} f(|x|_p)|x|_p^{it}d^*\mu_p(x)=\sum_{k\in \mathbb{Z}}f(p^{-k})p^{-ikt}\\
        &\mapsto \sum_{k\in \mathbb{Z}}f(p^{-k})\hat{T}_{\log p^{-k}}\in L^2(\bohrc{}),
    \end{split}
\end{equation*}
More generally, for $\sigma\geq0$,
\[\chi\mapsto \mathcal{M}_p(f)(\sigma,\chi)=\sum_{k=-\infty}^{\infty}f(p^{-k})p^{-\sigma k}\chi({T}_{log p^{-k}})\in L^2(\bohrc{}),\]
if and only if 
\[\sum_{k=-\infty}^{\infty} |f(p^{-k})|^2 p^{-2\sigma k}=\int_{\mathbb{Q}_p^{*}} |f(|x|_p)|^2 |x|_p^{2\sigma}d\mu_{p}^*(x)<+\infty,\]
in such case we have 
\begin{equation}
    ||\mathcal{M}_p(f)(\sigma,\cdot)||_{\bohrc{}}=||f(\cdot)|\cdot|_p^{\sigma}||_{\mathbb{Q}_p^*}.
\end{equation}
Note that this gives us a family of square-integrable complex random variables $\mathcal{M}_p(f)(\sigma,\chi)$ indexed by $f\in L^2(\mathbb{Q}_p^{*},|x|_p^{2\sigma}d^*\mu_p)$, that is
\begin{lemma}
    For every $0\leq \sigma<1$ the space $L^2(\mathbb{Q}_p^{*},|x|_p^{\sigma}d^*\mu_p)$ is linearly isometric to a subspace of $L^2(\bohrc{})$. 
\end{lemma}
\noindent Now, let $M,N\in L^2( \bohrc{})$ two random variables. Then the following holds 
    \begin{enumerate}
        \item $E[M]=\langle M, \hat{T}_0\rangle_{\bohrc{}}$
        \item $Var[M]=||M||_{L^2(\bohrc{})}^2-|\Lambda(M)|^2$  
        \item The covariance of $M$ and $N$ satisfies, $Cov[M,N]=\langle M, N \rangle_{L^2(\bohrc{})}-\Lambda(M)\overline{\Lambda(N)}.$
    \end{enumerate}
In particular, for a random variable of the form $\mathcal{M}_p(f)(\sigma,\cdot)$ we have $$E[\mathcal{M}_p(f)(\sigma,\cdot)]=\langle \mathcal{M}_p(f)(\sigma,\cdot),\hat{T}_0\rangle_{L^2(\bohrc{})}=f(1).$$
\begin{remark}
More generally we have the following random variables. Let $L(s)=\sum_{n=1}^{\infty} a_n e^{-\lambda_n s}$ be a Dirichlet series, convergent in the half-plane $Re(s)>\sigma_0$, and denote by $L_n(s)$ its partial sum. Let $s=\sigma+it$, where $\sigma$ is fixed and such that $t\mapsto L_n(\sigma+it) $  converges to $t\mapsto L(\sigma+it)$ with respect to $||\cdot||_{\infty}$ , then it is clear that 
\[L= \sum_{n=1}^{\infty} a_n e^{-\lambda_n\sigma} \hat{T}_{-\lambda_n}, \]defines a continuous random variable on the Bohr space. 
\end{remark}
Let $f^*(x)=\frac{1}{|x|_p}f(\frac{1}{|x|_p})$ for $x\neq 0$.   $L^1(\mathbb{Q}_p^*,d \mu_p^*)^{\mathbb{Z}_p^{*}}$  equipped with the multiplicative convolution and the involution given by $f\mapsto |x|_p\overline{f}^{*}=:f^{\#}(x)$ is a *-algebra. The Mellin transform satisfies
\[\mathcal{M}_p(\overline{f}^*)(\sigma,\chi)=\overline{\mathcal{M}_p(f)(1-\sigma,\chi)},\]
 and for $\sigma=\frac{1}{2}$ we have
\[\mathcal{M}_p(\overline{f}^*)(\frac{1}{2},\chi)=\overline{\mathcal{M}_p(f)(\frac{1}{2},\chi)},\]
for $M_p(f)(\frac{1}{2},\cdot)$ continuous the map $f \mapsto \overline{f}^{*}$ is compatible with the involution on $C(\bohrc{})$.  
\begin{proposition}
    The space $L^1(\mathbb{Q}_p)^{\mathbb{Z}_p^*}$ is embedded in $C(\bohrc{})$, moreover the involutions of *-algebras are compatible. 
    \end{proposition}
\begin{proof}
    Consider the map $\lambda:L^1(\mathbb{Q}_p^*,d \mu_p^*)^{\mathbb{Z}_p^{*}}\rightarrow L^1(\mathbb{Q}_p^*,|x|_p^{1/2}d \mu_p^*)^{\mathbb{Z}_p^{*}}$, given by $\lambda(g)(x)=f(x)$, where $f(x)=|x|_p^{-1/2}g(x)$, then this map take the involution $g\mapsto g^{\#}$ to the involution $f\mapsto \overline{f}^{*}$. Moreover, the function $\mathcal{M}_p(f)(\frac{1}{2},\chi)$ is an element of the algebra $C(\bohrc{})$. Indeed, since 
\[\mathcal{M}_p(f)(\frac{1}{2}+it)=\sum_{k=-\infty}^{\infty} g(p^{-k})p^{-ikt}, \]
by the Weierstrass $M$-test, the series converges uniformly, and can be extended to a continuous function on $\bohrc{}$.
\end{proof}

\begin{remark}
    Moreover, by equation $(8)$ the map $\mathcal{M}_p(\cdot)(1/2,\cdot)\circ \lambda:L^2(\mathbb{Q}_p^{*},d\mu_p^{*})^{\mathbb{Z}_p^*}\rightarrow L^2(\bohrc{})$, is a linear isometry.
\end{remark}
The convolution and multiplication on $C(\bohrc{})$ are related by the well now convolution theorem:
\begin{proposition}
Let $f,g$ such that $\mathcal{M}_p(f)(\sigma,\cdot)$ and $\mathcal{M}_p(g)(\sigma,\cdot)$ exists. Then it holds true that 
\[\mathcal{M}_p(f*g)(\sigma,\cdot)=\mathcal{M}_p(f)(\sigma,\cdot)\mathcal{M}_p(g)(\sigma,\cdot).\]
\end{proposition}

\subsection{Interlude: The Bohr-Jessen limit theorem.}
The value-distribution theory is the life-work of Harald Bohr. In the 1930's, Bohr was able to prove the following fundamental result: let $R\subset \mathbb{C}$ be a rectangle in the complex plane, whose edges are parallel to the coordinate axes. Denote by $\mu_0$ the one-dimensional Lebesgue measure. From the half plane  $\sigma>\frac{1}{2}$, we remove all the points which have the same imaginary part as, and smaller real part than, one of the possible zeros or the pole of $\zeta$ in this region, and the remaining part we denote by $G$. Then the following holds, 

\begin{theorem}
   
    (\textbf{Bohr-Jessen} \cite{Bohrjessenlimit})  For any $\sigma>\frac{1}{2}$, the limit value
\[B(R,\sigma,\zeta)=\lim_{T\rightarrow \infty} \frac{1}{2T} \mu_0(\{t\in[-T,T] : \sigma+it \in G, \ \log \zeta(\sigma+it)\in R\}),\]
exists.
\end{theorem}
Bohr and Jessen themselves developed an advanced geometric theory of sums of closed convex curves whose results are used in the proof of Theorem 1. Many alternative proofs of this theorem have been established since then, but in this small interlude we are interested in the techniques used by Laurincikas  and Bachchi. Laurincikas was able to obtain the limit theorem of the following form \cite{LimitLaurincikas}: define the probability measure $P_T$ on $\mathbb{C}$ by 
\[P_T(A,\sigma,\zeta)=\frac{1}{2T}\mu_0(\{t\in [-T,T]: \zeta(\sigma+it)\in A) \}\]
for any Borel set $A$ of $\mathbb{C}$. Then
\begin{theorem}
    For any $\sigma>\frac{1}{2},$ $P_T(A,\sigma,\zeta)$ is weakly convergent to a certain probability measure $Q(A,\sigma,\zeta)$ as $T\rightarrow \infty$. 
\end{theorem}
The probability measure $Q(A,\sigma,\zeta)$ can be given explicitly using the random variable (3), in fact we have that (\cite[Notes Chapter 5]{LimitLaurincikas}).
\begin{equation}
    \mu_{\Omega'}(\{\omega \in \Omega': \zeta(\sigma,\omega)\in A)\})= Q(A,\sigma,\zeta).
\end{equation}
\subsection{The distribution of $\zeta(\sigma,\chi)$}
Let $s\in \{s\in\mathbb{C}:s=\sigma+it,\sigma>1\}$. Then the Riemann zeta function is given by the following uniformly convergent series, 
\[\zeta(s)=\sum_{k=1}^{\infty} \frac{1}{n^s},
\]
then the assignment $t\mapsto \zeta(\sigma+it)$ defines an almost periodic function for each $\sigma>1$. Its extension to $C(\bohrc{})$ is then given by 
\[\chi\mapsto \zeta(\sigma,\chi)=\sum_{k=1}^{\infty} n^{-\sigma}\chi(e^{-it\log n}).\]
On the other hand, the behavior of this function in $\frac{1}{2}\leq\sigma <1$ changes dramatically, that is, for $\sigma$ in this region we do not have a continuous extension. Nevertheless, we still have an extension for $\frac{1}{2}<\sigma<1$ as an $L^2(\bohrc{})$-function. In order to prove this, we begin by stating the following theorem about on the convergence of a series of orthogonal random variables \cite[Theorem 2.9]{LimitLaurincikas}
\begin{theorem}
    Let $X_1,X_2,...$ be orthogonal random variables and
\[\sum_{k=1}^{\infty}E|X_m|^2 \log^2 k<\infty.\]
 Then the series 
\[\sum_{k=1}^{\infty}X_m\]
converges almost surely. 
\end{theorem}

\begin{lemma}
    $\zeta(\sigma,\chi)$ is a complex-valued random variable on the Bohr space. 
\end{lemma}
\begin{proof}
Define $\varphi_k(\chi)=\frac{\chi(e^{-it\log k})}{k^{\sigma}}$. Then by the prime factorization theorem and the orthogonality of $T_{-\log p}$ for different primes $p$, the variables $\varphi_k$ are orthogonal. On the other hand, we have that $E|\varphi_k|^2=\frac{1}{k^{2\sigma}}$. Hence, by Theorem 3, $\zeta(\sigma,\chi)$ converges almost surely. 
\end{proof}

Recall from equation (4), that a similar random variable was defined on the probability space $\Omega'$ given by $\zeta(\sigma,\omega)=\prod_{p<\infty } \left(1+\sum_{k=1}^{\infty}\frac{\omega(p)^{k}}{p^{k\sigma}}\right)$.  In fact, we will prove that $\zeta(\sigma,\chi)$ is the pull-back of the complex-valued random variable $\zeta(\sigma,\omega)$ on $\Omega'$. For this, let $\gamma=\{\gamma_1,..,\gamma_n\}$ be an algebraically independent set. Let $G_{\gamma}$ denote the subgroup of $(\mathbb{R},+)$ freely generated by the set $\gamma$. We define 
\[\mathbb{R}_{\gamma}:=Hom[G_{\gamma},S^1],\]
as the set of all homomorphisims from $G_{\gamma}$ to the multiplicative group $S1$.  Each $\mathbb{R}_{\gamma}$ is homeomorphic with the topological space $(S^1)^{n}$, since any element $\chi\in \mathbb{R}_{\gamma}$ is uniquely determined by its values on the generators $\gamma_1,\dots \gamma_n$.  Recall that, the Bohr compactification is equivalently defined by $\bohrc{}=Hom[\mathbb{R},S^1]$. This topological space can be identified with the projective limit of compact spaces $\{\mathbb{R}_{\gamma}\}_{\gamma}$, where $\gamma$ runs over all the finite algebraic independent sets of $\mathbb{R}$. Hence, we have continuous projection maps $\rho_{\gamma}:\bohrc{}\rightarrow \mathbb{R}_{\gamma}$, given by $\chi \mapsto \chi |_{G_{\gamma}}$, or equivalently, $\chi \mapsto \{\chi(\gamma_i)\}_{i=1}^{n}\in (S^1)^{n}$. Consider now the push-forward measures  $\rho_{\gamma,*}\mu_{\bohrc{}}$ on $(S^1)^{n}$, then it holds true that these measures coincide with the Haar measure on the compact group $(S^1)^n$ equipped with coordinate-wise multiplication \cite{bohrrelativ}. 
\begin{theorem}
    Let $\sigma>\frac{1}{2}$ and $A$ be a Borel set of $\mathbb{C}$. Then the distribution of $\zeta(\sigma,\chi)$ satisfies
\begin{equation}
    \zeta(\sigma,\chi)_{*} \mu_{\bohrc{}}(A)=\lim_{T\rightarrow \infty}\frac{1}{2T}\mu_0(\{t\in [-T,T]: \zeta(\sigma+it)\in A) \},
\end{equation}
where $\mu_0$ denotes the Lebesgue measure of $\mathbb{R}$.
\end{theorem}
\begin{proof}
Define the map $\rho_{\infty}:\bohrc{}\rightarrow \Omega'$, where the space $\Omega'$ is defined in (4), given by $\chi\mapsto (\chi(e^{it\log p_k}))_{k=1}^{\infty}$ where $p_k$ runs over the primes. Let $A$ be a cylindrical set of $\Omega'$, that is 
\[A=A_{I}\times\prod_{s\notin I} \mathbb{T}_s, \hspace{0.5cm} A_{I}=\prod_{t\in I}A_t\]
where $I$ is a finite set of primes, and $A_t\subset \mathbb{T}$, then the following holds
\[\rho_{\infty}^{-1}(A)=(\pi_I\circ \rho_{\infty})^{-1}(A_I),\]
where, $\pi_I:\Omega'\rightarrow \prod_{t\in I}\mathbb{T}_t$ is the canonical projection.  Note that $\pi_I\circ \rho_{\infty}=\rho_{\gamma}$, where $\gamma=\{\log p : p\in I\}$ is an algebraic independent set of $\mathbb{R}$, that is, $\rho_{\gamma}$ is the continuous projection from $\bohrc{}$ to $\mathbb{R}_{\gamma}$. Therefore $\rho_{\infty}^{-1}(A)$ is an open set. Since cylindrical sets form a basis for the topology of $\Omega'$, the map $\rho_{\infty}$ is continuous. Consider now an element $b\in \Omega'$, and an element $\hat{\chi}\in \bohrc{}$ such that $\hat{\chi}(e^{it\log p_k})=\overline{b_k}$. Then we have the following
\[\hat{\chi}\rho_{\infty}^{-1}(bA)=\rho_{\infty}^{-1}(A),\]
therefore,
\begin{equation}
    \rho_{\infty,*}\mu_{\bohrc{}}(bA)=\mu_{\bohrc{}}(\hat{\chi}\rho_{\infty}^{-1}(bA))=\rho_{\infty,*}\mu_{\bohrc{}}(A).
\end{equation}
Furthermore, we have that this is a normalized measure, therefore the measure $\rho_{\infty,*}\mu_{\bohrc{}}$ is the normalized Haar measure $\mu_{\Omega'}$ on $\Omega'$. Hence, for a random variable $X:\Omega'\rightarrow \mathbb{C}$, we can construct the pull back $X^*:\bohrc{}\rightarrow \mathbb{C}$, as $X\circ \rho_{\infty}$. It is easy to see that the distributions coincide in virtue of (11). Since $\chi \mapsto \zeta(\sigma,\chi)$, $\chi \in\bohrc{}$, is the pull back of the random variable $\omega \mapsto \zeta(\sigma,\omega)=\sum_{k=1}^{\infty}\omega(k)k^{-\sigma},$ therefore by Theorem 2 and equation (9), the theorem is proved. 
\end{proof}
The function $\bohrc{}\ni \chi \mapsto \zeta(\sigma,\chi)$ for $\frac{1}{2}<\sigma<1$ can be consider as an off-diagonal version of the Riemann zeta function, since we have that $\zeta(\sigma,\chi)$ can be identified by: 
\[Hom_{cont}(\mathbb{A}_f^*/\hat{\mathbb{Z}},S^1)=\prod_{p<\infty} i\mathbb{R}/\frac{2\pi i}{\log p}\mathbb{Z} \ni \chi \mapsto \prod_{p<\infty} \zeta_p(\sigma,\chi), \]
where \[\zeta_p(\sigma,\chi)=\int_{\mathbb{Q}_p^*}\phi_p(x)|x|_p^{\sigma}\chi_p(x)d^{*}\mu_p(x),\]
and $\chi_p\in i\mathbb{R}/\frac{2\pi i}{\log p}\mathbb{Z}$. \newline \\
Then limit (10) relates the off-diagonal zeta function with the Riemann zeta function. The behavior of $\lim_{\sigma \rightarrow \frac{1}{2}}\zeta(\sigma,\chi)$ is still an open problem \cite{takanobu}.  Nevertheless, the extension of the logarithmic derivative of $\zeta(s)$ to $\bohrc{}$ in the critical line has indeed a well-defined relation to the zeta function, as we will show in the next sections. 
\subsection{Probabilistic interpretation of the local formula and spectral measures}
The logarithmic derivative of the zeta function for $Re(s)>1$
is given by
\[-\frac{\zeta'(s)}{\zeta(s)}=\sum_{n=1}^{\infty} \frac{\Lambda(n)}{n^s}=\sum_{p<\infty} \log p \sum_{k=1}^{\infty} p^{-ks}, \]
for fixed $\sigma=Re(s)$ we can extend this function to an element of $C(\bohrc{})$. For $Re(s)>0$,  we can construct the $p$-adic family of random variables given by the extension of the functions
\[\frac{d}{ds}\log \zeta_p(s)=-\log p \sum_{k=1}^{\infty}p^{-ks}.\]
to $C(\bohrc{})$, these extensions are given by the following $p$-adic Mellin transforms  
\[\mathcal{M}_p(\Lambda(|x|_p^{-1}))(\sigma,\chi) =\log p \sum_{k=1}^{\infty} p^{-k\sigma}\hat{T}_{\log p^{-k}}(\chi)=\sum_{k=1}^{\infty}\frac{\Lambda(p^k)}{p^{k\sigma}}\hat{T}_{\log p^{-k}}(\chi).\]
Note that 
\[E[\mathcal{M}_p(\Lambda(|x|_p^{-1})(\sigma,\cdot)]=0,\]
hence for any $p$-adic random variable $\mathcal{M}_p(f)(\sigma,\cdot)$ we have,
\begin{equation}
    Cov[M_p(f)(\sigma,\cdot),\mathcal{M}_p(\Lambda(|x|_p^{-1}))(\sigma,\cdot)]=\log p \sum_{k=1}^{\infty}f(p^{-k})p^{-2k\sigma},
\end{equation}
and 
\begin{equation}
    Cov[\mathcal{M}_p(f)(\sigma,\cdot),\mathcal{M}_p(\Lambda^*(|x|_p^{-1}))(\sigma,\cdot)]=\log p \sum_{k=1}^{\infty} f(p^{k}) p^{-k(2\sigma-1)}.
\end{equation}
Consider now (formally) the off-diagonal logarithmic derivative given by 
\[\chi\mapsto\sum_{p<\infty} \mathcal{M}_p(\sigma,\Lambda(|x|_p^{-1}))(\chi), \]
In a similar way we extend this function to a random variable for $\frac{1}{2}<\sigma<1$. For this we need the following result, the proof can be found in \cite[Section 17.3]{Loebeprob}. 
\begin{proposition}
    Let $\{X_n\}_{n=1}^{\infty}$ be a sequence of real-valued random varibles defined on a probability space $(\Omega,\mathcal{F},P)$. Suppose that 
    \begin{itemize}
        \item $\{X_n:n=1,2,...\}$ are independent,
        \item $X_n$ is square-integrable, that is, $E[X_n^2]<\infty$ $(\forall n)$,
        \item $\sum_{n=1}^{\infty}Var(X_n)<\infty$, where $Var(X_n)=E[(X_n-E[X_n])^2]$. 
    \end{itemize}
    Then $\sum_{n=1}^{\infty}(X_n-E[X_n])$ is convergent $P$-almost everywhere; that is, $\sum_{n=1}^{N}(X_n-E[X_n])$ is convergent as $N\rightarrow \infty$ $P$-almost everywhere.  
\end{proposition}
\begin{proposition}
    For $\sigma>\frac{1}{2}$ the series  $\chi\mapsto\sum_{p<\infty} \mathcal{M}_p(\Lambda(|x|_p^{-1}))(\sigma,\chi),$ is convergent almost everywhere on $\bohrc{}$. 
\end{proposition}
\begin{proof}
    The function $\mathcal{M}_p(\Lambda(|x|_p^{-1}))(\sigma,\chi)$ is clearly an element of $L^2(\bohrc{})$. It holds true that 
\[\sum_{n=1}^{\infty}Var(\mathcal{M}_p(\Lambda(|x|_p^{-1}))(\sigma,\chi))=\sum_{p<\infty}(\log p)^2 \frac{p^{-2\sigma}}{1-p^{-2\sigma}} \]
since 
\[\frac{p^{-2\sigma}}{1-p^{-2\sigma}}\leq \frac{2^{2\sigma}}{2^{2\sigma}-1}\frac{1}{p^{2\sigma}},\]
the sum of variances converges for $\sigma>\frac{1}{2}$. \newline \\
The random variables $\{\mathcal{M}_p(\Lambda(|x|_p^{-1}))(\sigma,\chi)\}_{p<\infty}$ are independent by the proposition below. Therefore by Proposition 3 the result is proved. 
\end{proof}
\begin{definition}
    The real numbers $\lambda_1,...,\lambda_k$ are called rationally independent if they are linearly independent over $\mathbb{Z}$, i.e., for all $n_1,...,n_k\in \mathbb{Z}$ ,
\[\sum_{i=1}^{k} n_i \lambda_i =0 \implies n_1=\cdots n_k=0.\]
\end{definition}
The proof of the following Proposition can be found in \cite{kacstatistical}.
\begin{proposition}
    Let $\lambda_n$, $n \in \mathbb{N}$ be a sequence of rationally independent real numbers. Then the functions $\hat{T}_{\lambda_n}$ are stochastically independent under the Haar measure of $\bohrc{}$. 
\end{proposition}
Therefore, the sequence of elements $\hat{T}_{\log p}$ are stochastically independent under the Haar measure of $\bohrc{}$. \newline

We denote by $\frac{\zeta'}{\zeta}(\sigma,\chi)$ the corresponding extension of the logarithmic derivative given in Proposition 4. Now we study the covariance map attached to this random element. That is, for a random variable $X(\chi)$ we are interested in the map 
\[X\mapsto Cov[X(\chi),Re\logz{}(\sigma,\chi)].\]
Since $E[\logz{}(\sigma,\chi)]=0$, by the Riesz representation theorem, this is just the linear functional represented by $Re\logz{}$. Moreover, under certain restrictions, this covariance will be equal to a quadratic form attached to a bounded linear operator, or equivalently a spectral integral, which will be related to the Weil explicit sums. To begin with this, we start with the following result: We denote by $\mathcal{B}(L^2(\bohrc{}))$ the set of bounded linear operators on the Hilbert space $L^2(\bohrc{})$. 
\begin{theorem}
There exists a bounded linear operator  $\mathfrak{W}_p$ on $L^2(\bohrc{})$ such that its attached quadratic form  satisfies 

\begin{equation}
    \begin{split}
          \langle\mathfrak{W}_p \mathcal{M}_p(f)(1/2,\cdot),\mathcal{M}_p(f)(1/2,\cdot)\rangle &=Cov[\mathcal{M}_p(f*f^*)(1/2,\cdot),\mathfrak{W}_p \mathbf{1}]\\&=W_p(f*f^*),
    \end{split}
\end{equation}
and 
\begin{equation}
    E[\mathfrak{W}_p \mathcal{M}_p(f)(1/2,\cdot)]=W_p(f)
\end{equation}
\end{theorem}
\begin{proof}
    The Hilbert space $L^2(\bohrc{})$ provides a faithful representation of the $C^*$-algebra $C(\bohrc{})$. In particular, for $T_k\in C(\bohrc{})$  its representative is given by 
\[\Pi(k)\psi(\chi)=\hat{T}_k(\chi)\psi(\chi), \ \psi\in L^2(\bohrc{}),\]
Consider the representative of $\mathcal{M}_p(\Lambda+\Lambda^*)(\frac{1}{2},\cdot)$, denoted by $\mathfrak{W}_p\in \mathcal{B}(\bohrc{})$, then the equations $(14)$ and $(15)$ are satisfied in virtue of $(12)$ and $(13)$. 
\end{proof}
A similar construction can be done for $\sigma>0$, generating a bounded operator $\mathfrak{W}_p^{\sigma}$.  For a function $f\in C_{c}^{\infty}(\mathbb{R}_{+})$ we define 
\[\mathcal{M}_{fin}(f)(\sigma,\chi):=\sum_{p<\infty}\mathcal{M}_p(f)(\sigma,\chi),\]
note that, since the support of $f$ is a compact set contained in $(0,\infty)$, this sum is finite and every $p$-adic Mellin transform is well defined. 
\begin{proposition}
    Let $f\in C_c^{\infty}(\mathbb{R}_{+})$, then  
\begin{equation*}
    \begin{split}
                &Cov[\mathcal{M}_{fin}(f)(\sigma,\cdot),2Re\logz{}(\sigma,\cdot)]=\\
              &\sum_{p<\infty}\left(\log p \sum_{k=1}^{\infty}f(p^{-k})p^{-2k\sigma}+\log p \sum_{k=1}^{\infty} f(p^{k}) p^{-k(2\sigma-1)}\right),
    \end{split}
\end{equation*}
and 
\[\lim_{\sigma\rightarrow\frac{1}{2}}Cov[\mathcal{M}_{fin}(f)(\sigma,\cdot),2Re\logz{}(\sigma,\cdot)]=W_{fin}(f).\]
\end{proposition}
\begin{proof}
    Since $E[\logz{}]=0$, we have that 
\[Cov[\mathcal{M}_{fin}(f)(\sigma,\cdot),2Re\logz{}(\sigma,\cdot)]=\langle\mathcal{M}_{fin}(f)(\sigma,\cdot),2Re\logz{}(\sigma,\cdot)\rangle,\]
by the orthogonality property of the Mellin transform for different prime numbers $p$, the following holds true
\[Cov[\mathcal{M}_{fin}(f)(\sigma,\cdot),2Re\logz{}(\sigma,\cdot)]=\sum_{p<\infty} \langle\mathcal{M}_{p}(f)(\sigma,\chi),\mathcal{M}_p(\Lambda+\Lambda^*)(|x|_p^{-1})(\sigma,\chi)\rangle\]the result follows from (12) and (13). The above sum is finite since the support of $f$ is a compact set, the limit follows immediately. 
\end{proof}
\begin{example}
 Let $X>1$. And consider $g(x)=1_{(1,X)}(x)$ and $g(1)=g(X)=\frac{1}{2}$. Then 
\[Cov[\mathcal{M}_{fin}(g)(\sigma,\cdot),2Re \ \logz{}(\sigma,\cdot)]=\sum_{1<n<X} \Lambda(n) + \frac{1}{2} \Lambda(X)=\psi(x)+\frac{1}{2} \Lambda(X)\]
\end{example}
As mentioned in section 1.2, it is enough to consider a finite set of primes in order to analyze the Weil explicit sum. This will allow us to attach a quadratic form and a spectral integral to the term $W_{fin}(f*\overline{f}^*)$. The archimedean contribution will be analyzed in the next section. Denote by $\mathbb{P}$ the set of primes. As mentioned before, every $f\in C_c^{\infty}(\mathbb{R}^+)$ has attached a finite set of primes $I\subset \mathbb{P}$ such that for $p\in \mathbb{P}\setminus I$, $f(p^k)=0$ for all $k\in \mathbb{Z}$. Consider then the class of functions with the same set $I$ denoted by $C_I$, therefore
\[f\in C_I: \mathcal{M}_{fin}(f)(\sigma,\chi)=\sum_{p\in I}\mathcal{M}_p(f)(\sigma,\chi).\]\begin{theorem}
For every finite set $I\subset \mathbb{P}$ there exists a bounded linear operator  $\mathfrak{W}_I$ on $L^2(\bohrc{})$ such that for any $f\in C_I$ its attached quadratic form  satisfies
\begin{equation*}
    \begin{split}
        \langle\mathfrak{W}_I \mathcal{M}_I(f)(1/2,\cdot),\mathcal{M}_I(f)(1/2,\cdot)\rangle &=Cov[\mathcal{M}_I(f*f^*)(1/2,\cdot),\mathfrak{W}_I \mathbf{1}]\\ &=\sum_{p<\infty}W_p(f*f^*),
    \end{split}
\end{equation*}
and 
\[ E[\mathfrak{W}_I \mathcal{M}_I(f)(1/2,\cdot)]=\sum_{p<\infty}W_p(f)\]
\end{theorem}

\begin{proof}
    Let $I\subset \mathbb{P}$ and consider the function defined by 
\[\bohrc{}\ni\chi\mapsto \sum_{p\in I}\mathcal{M}_p(\Lambda+\Lambda^*(|x|_p^{-1}))\left(\frac{1}{2},\chi\right)=:\mathcal{M}_I(\Lambda+\Lambda^*)(\chi).\]
Since this function is a finite sum of continuous functions it follows that, $\mathcal{M}_I(\Lambda+\Lambda^*)(\chi)\in C(\bohrc{})$. Therefore, we can consider the operator attached to $\mathcal{M}_I(\Lambda+\Lambda^*)(\chi)$, that is, for $F\in L^2(\bohrc{})$, $\mathfrak{W}_I F(\chi)=\mathcal{M}_I(\Lambda+\Lambda^*)(\chi)F(\chi)$. Let $f\in C_{I}$. Note that 
\begin{equation}
    \begin{split}
        &\langle \mathfrak{W}_I \mathcal{M}_{fin}(f)(1/2,\chi),\mathcal{M}_{fin}(f)(\chi)(1/2,\chi)\rangle=\\ &\int_{\bohrc{}}\mathcal{M}_{I}(\Lambda)(\chi)\mathcal{M}_{fin}(f)(\chi)\mathcal{M}_{fin}(\overline{f}^*)(\chi)d\chi + \\  &\int_{\bohrc{}}\mathcal{M}_I(\Lambda^*)(\chi)\mathcal{M}_{fin}(f)(\chi)\mathcal{M}_{fin}(\overline{f}^*)(\chi)d\chi
    \end{split}
\end{equation}
The function $\mathcal{M}_{I}(\Lambda)(\chi)\mathcal{M}_{fin}(f)(\chi)\mathcal{M}_{fin}(\overline{f}^*)(\chi)$ has a series representation in the basis $\{\hat{T}_{r}\}_{r\in \mathbb{R}}$ with not necessarily zero coefficients only in the elements of the form $\hat{T}_{l\log p+m\log q +n\log r}$ where $p,q,r\in I$ and $l,m,n\in \mathbb{Z}$. Since $\{\log p\}_{p\in\mathbb{P}}$ is a rationally independent set, $l\log p+m\log q +n\log r=0$ if and only if $l=m=n=0$ or $p=q=r$ and $l+m+n=0$. Therefore, 
\begin{equation}
    \begin{split}
        \langle \mathfrak{W}_I \mathcal{M}_{fin}(f)(1/2,\chi),\mathcal{M}_{fin}(f)(\chi)(1/2,\chi)\rangle&=\sum_{p\in I}\langle\mathfrak{W}_p \mathcal{M}_p(f),\mathcal{M}_p(f)\rangle\\
        &=W_{fin}(f*\overline{f}^*)
    \end{split}
\end{equation}
where the last equality follows from Theorem $5$. Finally, note that by a similar argument as in proposition $6$ and since $\mathfrak{W}_I\hat{T}_0(\chi)=\mathfrak{W}_I\textbf{1}(\chi)=\mathcal{M}_{I}(\Lambda+\Lambda^*)(\chi)$, it holds true that
\[Cov[\mathcal{M}_I(f*f^*)(1/2,\cdot),\mathfrak{W}_I \mathbf{1}]=\sum_{p<\infty}W_p(f*f^*).\]
\end{proof}
We end this section with the following result about the spectral measure attached to the operator $\mathfrak{W}_I$. 
\begin{theorem}
    Let $I\subset \mathbb{P}$ be a finite set of primes. And let $\sigma(\mathfrak{W}_I)$ denote the spectrum of the operator $\mathfrak{W}_I$. Then for $f\in C_I$ the finite Weil sum $W_{fin}(f*\overline{f}^{*})$ can be expressed as 
    \begin{equation*}
        W_{fin}(f*\overline{f}^*)=\int_{\sigma(\mathfrak{W}_I)}\lambda \ dw_I(\lambda,\mathcal{M}_{fin}(f)),
    \end{equation*}
    Where $dw_I$ the projection valued measure on $\sigma(\mathfrak{W}_I)$ attached to $\mathfrak{W}_I$. 
\end{theorem}
\begin{proof}
    Since the operator $\mathfrak{W}_I$ is bounded and self adjoint by construction, its spectrum is real, moreover,  $\mathfrak{W}_I$ is a multiplication operator therefore its spectrum $\sigma(\mathfrak{W}_I)$ is given by the range of the continuous function $\mathcal{M}_{I}(\Lambda+\Lambda^*)$, that is 
\[\sigma(\mathfrak{W}_I)=\left\{\ \sum_{p\in I}2\log p \ Re \frac{p^{-1/2}\chi(T_{\log p^{-1}})}{1-p^{-1/2}\chi(T_{\log p^{-1}})}:\chi \in \bohrc{}\right\}.\]
By the spectral Theorem for self-adjoint bounded operators, there exists a unique projection-valued measure $dw_I$ on $\sigma(\mathfrak{W}^{(1/2)}_p)$ with values in $\mathcal{B}(L^2(\mathbb{B}))$ such that
\[\mathfrak{W}_I=\int_{\sigma(\mathfrak{W}_I)}\lambda \ dw_I(\lambda).\]
Moreover, for every $F\in L^2(\bohrc{})$, and for every Borel set $E$ of $\sigma(\mathfrak{W}_I)$, the map
\[dw_I(E,F)=\langle F, dw_I(E)F\rangle,\]
is a measure on the spectrum of $\mathfrak{W}_I$. By the theory of functional calculus, since the operator is a multiplication operator, we can express this measure as
\[dw_I(E,F)=\int_{\mathcal{M}_{I}(\Lambda+\Lambda^*)^{-1}(E)} |F(\chi)|^2 d\chi\]
Hence, for $f\in C_I$ we have the following
\[\sum_{p<\infty}W_p(f*\overline{f}^*)=\int_{\sigma(\mathfrak{W}_I)}\lambda \ dw_I(\lambda,\mathcal{M}_{fin}(f))\]

\end{proof}
\section{The Archimedean Place.}
In this section we study the Archimedean contribution to the Weil explicit sum. Inspired by the ideas of S. Haran in \cite{Mystery}, we use the techniques in $q$-calulus to get a discretization of $W_{\infty}$. It turns out that the Bohr compactification is a natural space to develop a probability theory of the the discretization of  Weil explicit sums, and in general, the tools in $q$-calculus are well defined in this space. This allows us to extend the ideas already presented, including this time the archimedean place.  
\subsection{The $q$-Mellin transform and $q$-random variables}
Let $q\in (0,1)$ be fixed. Consider the group denoted by $\widehat{G}_q:=i\mathbb{R}/\frac{2\pi i}{\log q^{-1}}\mathbb{Z}$. Note that, for $p\in \mathbb{P}$ and  $q=p^{-1}$, we recover the $p$-adic spaces $\widehat{\mathbb{Q}_p^* / \mathbb{Z}_p^*}$. In general, this group corresponds to the Pontryagin dual of the group $G_q:=q^{\mathbb{Z}}$. As before, we have that $L^2(\widehat{G}_q)$, and consequently $L^2(G_q)$, is a subspace of $L^2(\bohrc{})$ for every $q\in(0,1)$. Let $f\in L^1(G_q)\cap L^2(G_q)$, and consider its Fourier transform given by 
\[\chi \mapsto \int_{G_q} f(x)\chi(x)d_qx,\]
where $\chi\in \widehat{G}_q$. Therefore, we have the expression  
\[\int_{G_q} f(x)\chi(x)d_q^*x=\sum_{k\in \mathbb{Z}} f(q^k)q^{itk}\]
where $\chi(q)=q^{it}$, and $d^*_qx$ denotes the discrete measure on $G_q$. As before, we define the Mellin transform (for a suitable function $f$) at the point $s=\sigma+it$ as 
\[\mathcal{M}_q(f)(s):=\int_{G_q}f(x)x^sd_q^*x=\sum_{k\in \mathbb{Z}}f(q^k)q^{sk}\]
For suitable functions $f$, the function 
\[t\mapsto\mathcal{M}_q(f)(\sigma+it),\]
can be extended to a function in $L^2(\bohrc{})$ or $C(\bohrc{})$. In any case, its extension will be called the $q$-Mellin transform or $q$-random variable attached to $f$. This is given by
\[\bohrc{}\ni\chi\mapsto\mathcal{M}_q(\sigma,f)(\chi):=\sum_{k\in\mathbb{Z}}f(q^k)q^{\sigma k}\chi(q^{itk}).\]
Define the measure $d_qx:=(1-q)xd^*_qx$, and construct the function
\[m_q(f)(s):=\int_{G_q}f(x)x^{s-1}d_qx=(1-q)\sum_{k\in \mathbb{Z}}f(q^k)q^{sk}.\]
We have that $\lim_{q\rightarrow 1}m_q(f)=\mathcal{M}(f)(s)$, where $\mathcal{M}(f)(s)$ is the usual Mellin transform. As in the $p$-adic case, the following holds true: 
\[\mathcal{M}_q(\sigma,\overline{f}^*)(\chi)=\overline{\mathcal{M}_q(f)(1-\sigma,\chi)},\]
And for $\sigma=\frac{1}{2}$, the involution is preserved 
\[\mathcal{M}_q(\overline{f}^*)(1/2,\chi)=\overline{\mathcal{M}_q(f)(1/2,\chi)}.\]
Now we are ready to study the Archimedean contribution. 
\subsection{The $q$-digamma and the $q$-limit.}
Now we present an approach to establish the Archimedean contribution of the Weyl explicit sum as the limit of quadratic forms attached to operators defined in $L^2(\mathbb{B})$. The archimedean contribution can be written as
\begin{equation}
    W_{\infty}(f)=\log \pi f(1)-\frac{1}{2\pi i} \int_{1/2-i\infty}^{1/2+i\infty} \Re \left(\frac{\Gamma'}{\Gamma}(s/2)\right)\mathcal{M}(f)(s)ds,
\end{equation}
where $\mathcal{M}(f)$ is the complete Mellin transform of $f$. The idea is to represent this integral as a $q-$limit for $q\rightarrow 1$.  For this we introduce the following function which is the $q$-analogue of the digamma function $\frac{\Gamma'}{\Gamma}$introduced in \cite{digamma}.  For $0<q<1$ we define,
\[\psi_q(s):=-\log(1-q)+\log q \sum_{n=0}^{\infty} \frac{q^{n+s}}{1-q^{n+s}}.\]
It was established in \cite{digamma} that 
\[\lim_{q\rightarrow 1} \psi_q(s)=\frac{\Gamma'}{\Gamma}(s),\]
for $s\neq 0,-1,-2,...$ Hence this gives a $q$-analogue of the digamma function. We define the $q$-zeta local factor as $\zeta_q(s):=\frac{1}{1-q^{s}}$ for $s\neq  \frac{2\pi k}{\log q}$ and $k\in \mathbb{Z}$. Hence, we have
\[\frac{d}{ds}\log \zeta_q(s)=\log q \frac{q^{s}}{1-q^{s}}.\]
For a prime $p$ and  $q=p^{-1}$, we have the usual log-derivative of the local factor. We can rewrite the log-zeta interpolation as 
\[\psi_q(s)=-\log(1-q)+ \sum_{n=0}^{\infty}\frac{d}{ds}\log \zeta_q(s+n).\]
Let $0<\sigma<1$. Then the function $t\mapsto \frac{d}{ds}\log \zeta_q(\sigma+it)$ is a function in $\overline{\mathcal{C}}$. Define the $q$-\textit{von Mangoldt} function as $\Lambda_q(x)=\log q^{-1}$ if $x=q^{-k}$ for $k\in \mathbb{N}$ and zero otherwise. Hence its $q$-Mellin transform satisfies
\[-\mathcal{M}_q^{s}(\Lambda_q)=\frac{d}{ds}\log \zeta_q(s).\]
Therefore, the logarithm of the $q$-zeta local factor is a $q$-random variable on the Bohr space given by the function $\mathcal{M}_p^{\sigma}(\Lambda_q)(\chi)$.  For fixed $q$, let $\delta>0$ such that $q<\delta<1$. Then we have
\[\left|\frac{q^s}{1-q^s}\right|\leq \frac{1}{1-\delta} q^{Re(s)}.\]
Therefore, for fixed $Re(s)=\sigma>0$ and by the Weierstrass M-test we have that the function 
\[t\mapsto \sum_{n=0}^{\infty} \frac{d}{ds}\log \zeta_q(\sigma+n+it), \]
is a well-defined function in $\overline{\mathcal{C}}$. Therefore it can be extended to a function on $C(\bohrc{})$. Consider the function  $\psi_{q^2}(\frac{s}{2})$. Then, we have the following expansion at $\sigma=\frac{1}{2}$, 
\[\psi_{q^2}(\frac{1}{4}+it/2)=-log(1-q^2)+\log q^2\sum_{k=1}^{\infty}q^{ikt}\frac{q^{k/2}}{1-q^{2k}},\]
Since we are interested in $-\Re(\frac{\Gamma}{\Gamma'})$, we consider the following: 
\[-\Re \psi_{q^2}(\frac{1}{4}+it/2)=log(1-q^2)+\log q^{-1}\sum_{k\neq 0}q^{ikt}\frac{q^{|k|/2}}{1-q^{2|k|}}. \]
We denote by $H_q(\chi)\in C(\bohrc{})$ the extension of the almost periodic function $-\Re\psi_q^2(1/4+it/2)$ to $\bohrc{}$. Consider the following $q$-version of (18)
\[\int_{\bohrc{}} \mathcal{M}_q^{1/2}(f)(\chi)H_q(\chi)d\chi=f(1)\log(1-q^2)+(\log q^{-1}) \sum_{n\neq 0} \frac{1}{1-q^{2|n|}}\min(q^{-2n},1) f(q^{n}),\]
it turns out, that the expression on the right, coincides with one of the discretizations given by S. Haran in \cite{Mystery}, where it is showed that the RHS converges to the integral in equation (18) as $q\rightarrow 1$ . Moreover, if we define $\tilde{H}_q(\chi)=H_q(\chi)+\log \pi$, then    
\begin{equation}
    \int_{\bohrc{}} \mathcal{M}_q(f)(1/2,\chi)\tilde{H}_q(\chi)d\chi\rightarrow W_{\infty}(f)
\end{equation}
as $q\rightarrow 1$, \cite{Mystery}. This allows us to have the following Theorem: 
\begin{theorem}
    For every $I\subset \mathbb{P}$ there exists a bounded linear operator $\mathfrak{W}_{I,q}$ on $L^2(\bohrc{})$ such that for any $f\in C_I$ its attached quadratic form satisfies for $\mathcal{M}_{I,q}(f)(\chi):=\mathcal{M}_{fin}(f)(\chi)+\mathcal{M}_q(f)(1/2,\chi)$ the following equality
\[\langle\mathfrak{W}_{I,q} \mathcal{M}_{I,q}(f),\mathcal{M}_{I,q}(f)\rangle=Cov[\mathcal{M}_{I,q}(f*\overline{f}^*),\mathfrak{W}_{I,q}\textbf{1}]+\log((1-q^2)\pi)f*\overline{f}^*(1) \]
Moreover,
\[Cov[\mathcal{M}_{I,q}(f*\overline{f}^*),\mathfrak{W}_{I,q}\textbf{1}]+\log((1-q^2)\pi)f*\overline{f}^*(1)\rightarrow W(f*\overline{f}^*)\]
As $q\rightarrow 1$.
\end{theorem}
\begin{proof}
    Consider the multiplication operator attached to the function $\tilde{H}_q(\chi)$, denoted by $\mathfrak{W}_q$. Define the operator $\mathfrak{W}_{I,q}:=\mathfrak{W}_q+\mathfrak{W}_I$, where $\mathfrak{W}_I$ is defined in Theorem $6$. Since $I$ consists of finitely many primes, we can construct a sequence $q\rightarrow 1$ such that $\{\log q\}\cup \{\log p\}_{p\in I}$ is a rationally independent set. Then a similar argument as the one given in Theorem $6$ implies 
\begin{equation*}
    \begin{split}
        \langle\mathfrak{W}_{I,q} \mathcal{M}_{I,q}(f),\mathcal{M}_{I,q}(f)\rangle&=\sum_{r\in\{q\}\cup I } \langle\mathfrak{W}_{r} \mathcal{M}_{r}(f),\mathcal{M}_{r}(f)\rangle \\
&= W_{fin}(f*\overline{f}^*)+\langle\mathfrak{W}_{q} \mathcal{M}_{q}(f),\mathcal{M}_{q}(f)\rangle. 
    \end{split}
\end{equation*}
Now the result follows from limit (19).
\end{proof}
\begin{corollary}
        The following limit holds 
\[E[\mathcal{M}_{I,q}(f*\overline{f}^*)\mathfrak{W}_{I,q}\textbf{1}]\rightarrow W(f*\overline{f}^*)\]
as $q\rightarrow 1$. 
    \end{corollary}
It follows that the existence of a sequence $q\rightarrow 1$ such that $$E[\mathcal{M}_{I,q}(f*\overline{f}^*)\mathfrak{W}_{I,q}\textbf{1}]\leq 0,$$ implies the Riemann hypothesis. Or equivalently, following the observation of \cite{Mystery},  the Riemann hypothesis is equivalent to the bound 
\[Cov[\mathcal{M}_{I,q}(f*\overline{f}^*),\mathfrak{W}_{I,q}\textbf{1}]\leq ||f||_{L^2(G_q)}^2\log\left(\frac{(1-q)^{-1}}{2\pi}\right)\]
for a sequence of real numbers $q\in (0,1)$ convergent to $1$. 
Since the operator $\mathfrak{W}_{I,q}$ is a multiplication operator, its spectrum $\sigma(\mathfrak{W}_{I,q})$ is given by the range of 
\begin{equation}
    \begin{split}
        \bohrc{}\ni \chi\mapsto\sum_{p\in I}&\log p \ Re \frac{p^{-1/2}\chi(T_{\log p^{-1}})}{1-p^{-1/2}\chi(T_{\log p^{-1}})}\\ &+ log((1-q^2)\pi)  +\log q^{-1}\sum_{k\neq 0}\chi(T_{\log q})^{k}\frac{q^{|k|/2}}{1-q^{2|k|}}+\log \pi.
    \end{split}
\end{equation}
Therefore, by similar arguments in the proof of Theorem 7, we have the following result:
\begin{theorem}
    Let $I\subset \mathbb{P}$ be a finite set of primes. Then for $f\in C_{I}$ the Weil sum can be expressed as
    \begin{equation*}
W(f*\overline{f}^{*})=\lim_{q\rightarrow 1} \int_{\sigma(\mathfrak{W}_{I,q})}\lambda \ dw_{I,q}(\lambda,f)
    \end{equation*}
    where, the measure $d{w}_{I,q}(\cdot,f)$ only depend on the function $f$ and the number $q$.
\end{theorem}
\begin{proof}
    Given $f\in C_I$, by the spectral Theorem for self-adjoint bounded operators, there exists a measure  $dw_{I,q}(\cdot,f)$ such that 
\[\langle\mathfrak{W}_{I,q} \mathcal{M}_{I,q}(f),\mathcal{M}_{I,q}(f)\rangle=\int_{\sigma(\mathfrak{W}_{I,q})}\lambda \ dw_{I,q}(\lambda,f), \]
the result follows from Theorem 8. 
\end{proof}The positivity of these integrals will imply the Riemann Hypothesis. Let $f\in C_I$ and consider the linear functional defined on $C(\bohrc{})$ given by 
\[C(\bohrc{})\ni F \mapsto \int_{\bohrc{}} F(\chi) |\mathcal{M}_{I,q}(f*\overline{f}^*)(\chi)|^2d\chi. \]
Since this define a positive linear functional, by the Riesz-Markov Representation Theorem, there exists a unique Radon measure  $d_{f}\chi$ on $\bohrc{F}$. Suppose that $||\mathcal{M}_{I,q}(f*\overline{f}^*)||_{L^2(\mathbb{B})}=1$. Then it is clear that $d_f\chi$ defines a probability measure on $\mathbb{B}$. Consider the random variable defined by the function (20), that is $\mathfrak{W}_{q,I}\mathbf{1}(\chi)$ (this is a truncation of the real part of the logarithmic derivative of the zeta function viewed as a off-diagonal function or as an extension on $\bohrc{B}$). The probability distribution of $\mathfrak{W}_{q,I}\mathbf{1}(\chi)$ is then given by the spectral measure $dw_{I,q}(\cdot,f)$ of the operator $\mathfrak{W}_{q,I}$, and define a probability measure on the spectrum $\sigma(\mathfrak{W}_{I,q})$. As a corollary of the last result, we have the following. 

\begin{corollary}
    Let $\lambda$ be a random variable in the probability space \newline$(\sigma(\mathfrak{W}_{I,q}),\mathcal{B}(\sigma(\mathfrak{W}_{I,q})),dw_{I,q}(\cdot,f))$. Then 
    \[E[\lambda]=\langle\mathfrak{W}_{I,q} \mathcal{M}_{I,q}(f),\mathcal{M}_{I,q}(f)\rangle\]
\end{corollary}
\begin{remark}
    In \cite{Burnol}, J.F Burnol reproved S. Haran explicit formula \cite{Haran1990}. In an older paper \cite{Burnol2}, J. F. Burnol proposed the following strategy to prove the Riemann hypothesis: After a change of sign and defining $C(f*\overline{g}^*):=-W(f*\overline{g}^*)$ for $f,g\in C_{c}^{\infty}(\mathbb{R})$, it is proposed to find a generalized, stationary, zero mean stochastic process with $\mathcal{C}=\mathbb{A}^{*}/\mathbb{Q}^{*}$  as "time" whose covariance would be $C$. That is, to have a probability measure $\mu$ on the distributions on the classes of ideles denoted by $\mathcal{D}$ and the identity:\[\int_{\mathcal{D}} X_f(d)\overline{X_g(d)}d\mu(d)=C(f*\overline{g}^*),\]
where $f,g$ are considered as test-functions on $\mathcal{C}$ (for details see \cite{Haran1990}), and $X_f$ and $X_g$ are to associated "coordinates" on the space of distributions $d\mapsto X_f(d)=d(f)$ and $d\mapsto X_g(d)=d(g)$. So $C(f*\overline{g}^*)$ is the covariance of the random variables defined by $f$ and $g$. This implies the Riemann hypothesis since this shows that $-W(f*\overline{f}^*)$ is a variance and therefore is a positive number. For simplicity consider the case when $Supp \ f,Supp \ g \subset (\frac{1}{2},2)$. In this case the corresponding operator in theorem 8 is given by $\mathfrak{W}_{I,q}=\mathfrak{W}_q$, that is, the multiplication operator attached to the function $\tilde{H}_q(\chi)$, and $C(f*\overline{g}^*)$ is then given by the limit (Theorem 8)
\[\lim_{q\rightarrow 1} \int_{\mathbb{B}} \mathcal{M}_q(f)(1/2,\chi)\overline{\mathcal{M}_q(g)(1/2,\chi)} \tilde{H}_q(\chi) d\chi=-C(f*\overline{g}^*),\]
this is not a limit of variances since the linear functional 
$$C(\mathbb{B})\ni f\mapsto \int_{\mathbb{B}}f (-\tilde{H}_{q}(\chi))d\chi,$$
defines a signed measure and not a probability measure. On the other hand, for $q$ sufficiently close to $1$, the integrals of the above limit are positive, since in this case the positivity has been proved, \cite{connespositive}. In this work we showed that the explicit sums can be understanded within a probability framework. Nevertheless, the strategy of J.F Burnol remains open.

\end{remark}

\section{Outlook.}

The first step in constructing quantum mechanics from classical mechanics is given by the quantization rule that replaces the Poisson bracket of observables, i.e. functions on the phase space, by the commutator of operators on an abstract Hilbert space $\mathcal{H}$. This operators are the corresponding observables of quantum theory \cite{FLORESGONZALEZ2013394}. 
The corresponding algebra of operators (the so called Heisenberg algebra)  is represented on the Hilbert space $\mathcal{H}_{sch}$ of square-integrable functions $L^2(\mathbb{R},dx)$, this is called the Schrödinger's representation.  The celebrated Stone-von Neumann theorem establishes that any irreducible representation of the Heinsenberg algebra with certain commutation relations is unitarily equivalent to the Schrödinger's representation. Nevertheless, motivated by the physical principle of a discrete structure of space, one may relax one or some of the regularity assumpltions of the uniqueness theorem, for example, one can assume that one of the operators of the Heisenberg algebra is not weakly continuous. As a consequence, the Stone-von Neumann theorem does not hold, and therefore we obtain a non-unitarily equivalent representation, this is the so called polymer representation. Moreover, as a consequence of this, the operator $\hat{p}=-i\hbar \frac{d}{dx}$ is not well defined (but can be approximated!), which is expected in a discrete space. It turns out that the Hilbert space $\mathcal{H}_{poly}$ of such representation is given by $\mathcal{H}_{poly}=L^2(\bohrc{},d\chi)$. \newline
Polya and Hilbert proposed the idea that in order to understand the location of the zeros of the Riemann zeta function, one should find a Hilbert space $\mathcal{H}$ and an operator $D$ in $\mathcal{H}$ whose spectrum is given by the nontrivial zeros of the zeta function. This project is still carry on by, for example, the work of Alain Connes or Shai Haran. The Hilbert space $L^2(\mathbb{B})$ is directly related with two theories: the fundamental role as the Hilbert space in the Polymer representation, and as shown, for example,  in Theorem 4, with the theory of distributions of the zeta function initiated by Harald Bohr. Theorem 8, Corollary 2 and Theorem 9, shown how this space also allow us to have a probabilistic interpretation and an expression of the Weil explicit sum as a limit of integrals respect to spectral measures of operators defined on this space. The author believes that more information of the zeros of the zeta function and more possible interesting results will emerge on the study of the connections between PQM and the "off-diagonal" functions presented here.

%%%%%%%%%%%%%%%%
\bibliographystyle{plain}
\bibliography{biblio}

\end{document}